\documentclass[12pt,oneside]{article}
\usepackage{amsmath,amssymb,amsfonts,amsthm}
\usepackage{graphicx}
\usepackage{multirow}
\usepackage{rotating}
\newcommand{\set}[1]{\left\{#1\right\}}

\def\paa{\dot{\partial}}

\def\Section#1{\vspace{30truept}\addtocounter{section}{1}\setcounter{thm}{0}
\setcounter{equation}{0}{\noindent\Large\bf
    \arabic{section}.~~#1}\par \vspace{12pt}}

\newtheorem{thm}{Theorem}[section]
\newtheorem{cor}[thm]{Corollary}
\newtheorem{lem}[thm]{Lemma}
\newtheorem{prop}[thm]{Proposition}

\newtheorem{rem}[thm]{Remark}

\newcommand\overcirc[1]{\raisebox{10pt}{\small$\circ$}{\kern-7.5pt}\mbox{$#1$}}
\newcommand\undersym[2]{\raisebox{-6pt}{$#2$}{\kern-5pt}\mbox{$#1$}}
\newcommand\overdiamond[1]{\raisebox{10pt}{\small$\star$}{\kern-7.5pt}\mbox{$#1$}}
\newcommand\overast[1]{\raisebox{10pt}{\small$\ast$}{\kern-7.5pt}\mbox{$#1$}}
\newcommand\overlind[1]{\raisebox{10pt}{\small$\overline{{\hspace{2pt}}\star}$}{\kern-7.5pt}\mbox{$#1$}}
\newcommand\overlinc[1]{\raisebox{10pt}{\small$\overline{{\hspace{2pt}}\circ}$}{\kern-7.5pt}\mbox{$#1$}}
\newcommand\overlina[1]{\raisebox{10pt}{\small$\overline{{\hspace{1pt}}\ast}$}{\kern-7.5pt}\mbox{$#1$}}

\numberwithin{equation}{section}
\begin{document}
\title{\bf{Generalized   $\beta$-conformal change of Finsler metrics}\footnote{ArXiv Number: 0906.5369} }
\author{{\bf Nabil L. Youssef$^\dag$, S. H. Abed$^\ddag$
 and S. G. Elgendi$^\sharp$}}
\date{}
%\thanks{\it Department of Mathematics, etc}
%\pagestyle{fancy}

             % End of preamble and beginning of text.
\maketitle                     % Produces the title.
\vspace{-1.15cm}
\begin{center}
{$^\dag$Department of Mathematics, Faculty of Science,\\
Cairo University, Giza, Egypt}
\end{center}
\vspace{-0.8cm}
\begin{center}
 nlyoussef2003@yahoo.fr, nyoussef@frcu.eun.eg
\end{center}

\begin{center}
{$^\ddag$Department of Mathematics, Faculty of Science,\\
Cairo University, Giza, Egypt}
\end{center}
\vspace{-0.8cm}
\begin{center}
sabed52@yahoo.fr, sabed@frcu.eun.eg
\end{center}
\vspace{-0.5cm}
\begin{center}
{$^\sharp$Department of Mathematics, Faculty of Science,\\
Benha University, Benha,
 Egypt}
\end{center}
\vspace{-0.8cm}
\begin{center}
salahelgendi@yahoo.com
\end{center}
\smallskip

\vspace{1cm} \maketitle
\smallskip

%%%%%%%%%%%%%%%%%%%%%%%%%%%%%%%%%%%%Abstract%%%%%%%%%%%%%%%%%%%%%%%%%%%%%%%%%%%%%%
\noindent{\bf Abstract.} In this paper, we introduce  and investigate a general transformation or change
of Finsler metrics, which is referred  to as a
generalized $\beta$-conformal change:
$$L(x,y) \longrightarrow\overline{L}(x,y) = f(e^{\sigma(x)}L(x,y),\beta(x,y)).$$
This transformation combines both $\beta$-change  and conformal change in a
general setting. The change, under this transformation, of the fundamental
Finsler connections, together with their associated geometric objects, are obtained. Some invariants   and various
special Finsler spaces are investigated under this change. The most important
 changes of Finsler metrics existing in the literature are deduced from the generalized $\beta$-conformal change
as special cases.



\newpage
%%%%%%%%%%%%%%%%%%%%%%%%%%%%%%%%%Introduction%%%%%%%%%%%%%%%%%%%%%%%%%%%%%%%%%%

\vspace{30truept}\centerline{\Large\bf{Introduction}}\vspace{12pt}
\par
In the context of Riemannian geometry, there is not only a complete
local theory, but also a complete global theory, with many
practitioners working on both approaches. In this sense, Riemannian
geometry is indeed a complete theory. However, the situation in
Finsler geometry is substantially different. Finsler geometry was
first introduced locally by Finsler himself, to be studied by many
eminent mathematicians for its theoretical importance and
applications in the variational calculus, mechanics and theoretical
physics. Moreover, the dependence of the fundamental function $L(x,
y)$ on both the positional argument $x$ and directional argument $y$
offers the possibility to use it to describe the {\bf anisotropic}
properties of the physical space.
\par
Let $F^n=(M,L)$  be an n-dimensional Finsler manifold. For a
differential one-form $\beta(x,dx)=b_{i}(x)dx^{i}$ on $ M$, G.
Randers \cite{r1.11}, in 1941, introduced a  special Finsler space
defined by the change $\overline{L}=L+\beta$, where $L$ is
Riemannian, to consider a unified field theory.\footnote{In 1941, G. Randers published his paper \lq\lq On an
asymmetrical metric in the four-space of general relativity\rq\rq. In
this paper, Randers considered the simplest possible asymmetrical
generalization of a Riemannian metric. Adding a 1-form to the
existing Riemannian struture, he was the first to introduce a
special Finsler space. This space - which became known in the
literature as a Randers space - proved to be mathematically and physically very
important. It was one of the first attempts to study a physical
theory in the wider context of Finsler geometry, although Randers was not aware that the geometry he used was a special type of
Finsler geometry.} M. Masumoto \cite{r1.8}, in 1974, studied Randers
space and generalized Randers space in which $L$ is Finslerian. V.
Kropina \cite{r1.4} introduced the change $\overline{L}=L^2/\beta$,
where $L$ is Reimannian, which has been studied by many authors such
as Shibata \cite{r1.12} and Matsumoto \cite{r1.9}.  Randers  and
Kropina spaces are closely related to physics and so Finsler spaces
with these metrics have been studied by many authors, from various
standpoint in the physical and mathematical aspects (\cite{r1.3},
\cite{r1.5}, \cite{r1.10},  \cite{r1.14}, \cite{r1.15},
\cite{r1.16}). It was also applied to the theory of the
electron microscope by R. S. Ingarden \cite{Ingarden}. For a Kropina
space (the Finsler space equipped with Kropina's metric), there are
close relations between the Kropina metric and the Lagrangian
function of analytic dynamics \cite{r1.12}. In 1984, C. Shibata
\cite{r1.13} studied the general case of any $\beta$-change, that
is, $\overline{L}=f(L,\beta)$ which generalizes many changes in
Finsler geometry (\cite{r1.4}, \cite{r1.8}, \cite{geo.}). In this
context, he investigated the change of torsion and curvature tensors
corresponding to the above transformation. In addition, he also
studied some special Finsler spaces corresponding to specific forms
of the function $f(L, \beta)$.

On the other hand, in 1976, M. Hashiguchi \cite{r1.6} studied  the
conformal change of  Finsler metrics, namely, $ \overline{L} =
e^{\sigma(x)} L$. In particular, he also dealt with the special
conformal transformation named C-conformal.  This change has been
studied by many authors (\cite{r1.7}, \cite{r1.17}). In 2008, S.
Abed (\cite{r1.1}, \cite{r1.2}) introduced the transformation $
\overline{L} = e^{\sigma(x)} L + \beta$, thus generalizing the
conformal, Randers and generalized Randers changes. Moreover, he
established the relationships between some important tensors
associated with $(M, L)$ and the corresponding tensors associated
with $(M,\overline{L})$. He also studied some invariant and
$\sigma$-invariant properties and obtained a relationship between
the Cartan connection associated with $(M, L)$ and the transformed
Cartan connection associated with $(M, \overline{L})$.

 In this paper, we  deal with a general change of Finsler metrics  defined by:
\begin{equation*}\label{change} L(x,y)
\longrightarrow\overline{L}(x,y)=f(e^{\sigma(x)}L(x,y),\beta(x,y))=f(\widetilde{L},\beta),
\end{equation*}
where  $f$ is a positively homogeneous
function of degree one in $\widetilde{L}:= e^{\sigma}L$ and $\beta$. This change
will be referred to as a generalized $\beta$-conformal change. It is
clear that this change is a generalization of the above mentioned
changes and deals simultaneously with $\beta$-change and conformal
change. It combines also the special case of Shibata
($\overline{L}=f(L,\beta)$) and that of Abed
($\overline{L}=e^\sigma L+\beta$).

The present paper is organized as follows. In section 1,   the
relationship between the Cartan connection associated with $(M,
L)$ and the transformed Cartan connection associated with $(M,
\overline{L})$ is obtained  (Theorem \ref{gammA_2} ). The properties
that $\sigma$ being  homothetic,  $b_i$ being Cartan-parallel and the difference tensor being zero
are investigated (Theorems \ref{abd}, \ref{cases} and \ref{berwald}). The coefficients of the
fundamental linear connections of Finsler geometry are computed
(Theorem \ref{final.0}).

In section 2, the torsion and curvature tensor fields of the  fundamental linear
connections, corresponding to a generalized $\beta$-conformal change, are
obtained (Theorem \ref{final}). Some invariants are found
(Corollary \ref{final-1}) and some properties concerning certain
special Finsler spaces are investigated (Theorems \ref{landesberg},
\ref{minkowski.1} and \ref{minkowski.2}).

Finally, in section 3, many interesting changes of Finsler metrics  are
obtained  as special cases form the present change.

%%%%%%%%%%%%%%%%%%%%%% SECTION 0. Notations %%%%%%%%%%%%%%%%%%%%%%%%%%%%%%%
\vspace{30truept}\noindent{\Large\bf  Notations}\\

Throughout the present paper, $(M,L)$ denotes an n-dimensional
$C^\infty$ Finsler manifold; L being the fundamental Finsler
function. Let $(x^i)$ be the coordinates of any  point of the base
manifold M  and $(y^i)$ a supporting element at the same point. We
use the following
notations:\\
   $\partial_i$: partial differentiation with respect to $x^i$,\\
   $\dot{\partial}_i$:  partial differentiation
    with respect to  $y^i$ (basis vector fields of the vertical bundle),\\
   $g_{ij}:=\frac{1}{2}\dot{\partial}_i\dot{\partial}_j L^2=\dot{\partial}_i
   \dot{\partial}_jE$:  the
   Finsler metric tensor;  $E:=\frac{1}{2}L^2$:   the energy function,\\
   $l_i:=\dot{\partial}_iL=g_{ij} l^j=g_{ij}\frac{y^j}{L}$: the
    normalized supporting element; $l^i:=\frac{y^i}{L}$,\\
 $l_{ij}:=\dot{\partial}_il_j$,\\
   $h_{ij}:=Ll_{ij}=g_{ij}-l_il_j$:  the angular metric tensor,\\
   \vspace{7pt}$C_{ijk}:=\frac{1}{2}\dot{\partial}_kg_{ij}=\frac{1}{4}\dot{\partial}_i
    \dot{\partial}_j\dot{\partial}_k L^2$:  the Cartan  tensor, \\
   $G^i$: the components of the   canonical spray associated
    with $(M,L)$,\\
    $ N^i_j:=\dot{\partial}_jG^i$: the Barthel or Cartan nonlinear connection
    associated with $(M,L)$,\\
    $G^i_{jh}:=\dot{\partial}_hN^i_j=\dot{\partial}_h\dot{\partial}_jG^i$,\\
    \vspace{7pt}$\delta_i:=\partial_i-N^r_i\dot{\partial}_r$: the basis
     vector fields of the horizontal bundle,\\
      $C^i_{jk}:=g^{ri}C_{rjk}=\frac{1}{2}g^{ir}\dot{\partial}_kg_{rj}$:
     the h(hv)-torsion tensor,\\
     \vspace{7pt}$\gamma^i_{jk}:=\frac{1}{2}g^{ir}(\partial_jg_{kr}+\partial_kg_{jr}-\partial_rg_{jk})$:
        the Christoffel symbols with respect to $\partial_i$,\\
       $\Gamma^i_{jk}:=\frac{1}{2}g^{ir}(\delta_jg_{kr}+\delta_kg_{jr}-\delta_rg_{jk})$:
          the Christoffel symbols with respect to $\delta_i$,\\
          ${\quad\quad\!\!}=\gamma^i_{jk}+g^{it}(C_{jkr}N^r_{t}-C_{tkr}N^r_{j}-C_{jtr}N^r_{k})\vspace{7pt}$.

         We have:
\begin{description}
    \item[-]The canonical spray G:
    $G^h=\frac{1}{2}\gamma^h_{ij}y^iy^j$.
    \item[-]The Barthel connection N:
    $N^i_j=\dot{\partial}_jG^i=G^i_{jh}y^h=\Gamma^i_{jh}y^h$.
    \item[-]The Cartan connection $C\Gamma$:
    $(\Gamma^i_{jk},N^i_j,C^i_{jk})$.
     \item[-]The Chern (Rund) connection $R\Gamma$:
    $(\Gamma^i_{jk},N^i_j,0)$.
    \item[-]The Hashiguchi connection $H\Gamma$:
    $(G^i_{jk},N^i_j,C^i_{jk})$.
    \item[-] The Berwald connection $B\Gamma$:
    $(G^i_{jk},N^i_j,0)$.\vspace{0pt}
\end{description}

For a Finsler connection $(\Gamma^i_{jk},
         N^i_j,C^i_{jk})$, we define:\\
         $X^i_{j\mid k}:=\delta_kX^i_j+X^r_j\Gamma^i_{rk}-X^i_r\Gamma^r_{jk}$: the
    horizontal covariant derivative of $X^i_j$,\\
   $ X^i_j|_k:=\dot{\partial}_kX^i_j+X^r_jC^i_{rk}-X^i_rC^r_{jk}$: the
    vertical  covariant derivative of $X^i_j$.

\vspace{7pt}
Transvecting  with $y^j$ will be denoted by the subscript 0
(excluding $p_0, q_0,
 s_0$). For example, we write $B^i_{j0}$ for $B^i_{jk}y^k$.

\vspace{7pt}
Finally, the following special symbols will also be used:
\begin{description}
  \item[\textbf{-}] $\Theta_{(j,k,r)}\{A_{jkr}\}:=A_{jkr}-A_{krj}-A_{rjk}.$
  \item[\textbf{-}] $\mathfrak{A}_{(j,k)}\{A_{jk}\}:=A_{jk}-A_{kj}$: the alternative sum with
respect to the indices j and k.
\end{description}

%%%%%%%%%%%%%%%%%%%%%%%%%SECTION 2 Changes of connections%%%%%%%%%%%%%%%%%%%%%%%%%%%%%%
 \Section{\textbf{Changes of connections}}
  Let $F^n=(M,L)$  be an n-dimensional $C^{\infty}$ Finsler manifold with
fundamental function $L= L(x,y)$. Consider the following change of
Finsler structures  which will be referred to as  a generalized
$\beta$-conformal change:
\begin{equation}\label{change} L(x,y)
\longrightarrow\\\overline{L}(x,y)=f(e^{\sigma(x)}L(x,y),\beta(x,y)),
\end{equation}
where  $f$ is a positively homogeneous function of degree one in
$e^\sigma L$ and  $\beta$  and \,\, $\beta=b_i(x)dx^i$. Assume
that $\overline{F}^n=(M,\overline{L})$ has the structure of a
Finsler space. Entities related to $\overline{F}^n$ will be
denoted by barred symbols.

We define
   $$f_1:=\frac{\partial f}{\partial \widetilde{L}}\,,\qquad
   f_2:=\frac{\partial f}{\partial \beta}\,,\qquad
f_{12}:=\frac{\partial^2 f}{\partial \widetilde{L}\partial\beta},
\cdots, etc.,$$ where $\widetilde{L}=e^\sigma L$. We use the
following notations:
 \begin{eqnarray*}
   q&:=&ff_2,\hspace{3.8cm} p:=ff_1/L,\\
    q_0&:=&ff_{22},\hspace{3.5cm} p_0:=f^2_2+q_0,\\
    q_{-1}&:=&ff_{12}/L,\hspace{2.6cm} p_{-1}:= q_{-1}+pf_2/f,\\
    q_{-2}&:=&f(e^{\sigma}f_{11}-f_{1}/L)/L^2, \quad p_{-2}:=q_{-2}+e^\sigma
    p^2/f^2.
 \end{eqnarray*}
  Note that the subscript under the the above geometric objects indicates
the degree of homogeneity of these objects. We also use the
notations:
$$b^i=g^{ij}b_j,\quad m_i:= b_i-(\beta/L^2)y_i\neq 0, \quad
\sigma_i:=\partial_i\sigma, \quad p_{02}:=\frac{\partial
p_0}{\partial \beta}.$$

 The following lemmas  enable us to compute the
geometric objects
 associated  with  the space $\overline{F}^n$ obtained from $F^n$
  by a generalized $\beta$-conformal change.
 They can be proved by making use of Euler theorem
 of homogenous functions and the homogeneity properties   of
  $p,\,p_0,\,p_{-1},\,p_{-2}$; $q,\, q_0,\,q_{-1}, \,
 q_{-2}$.
\begin{lem}\label{f} The following identities hold:
\begin{description}
    \item[(a)] $e^\sigma L f_1+\beta f_2=f,$
    \item[(b)] $e^\sigma L f_{12}+\beta f_{22}=0,$
    \item[(c)] $ e^\sigma L f_{11}+\beta
     f_{12}=0$.
\end{description}
\end{lem}

\begin{lem}\label{relations}
The following identities hold:
\begin{description}
    \item[(a)]$ q_0\beta+e^{\sigma} q_{-1}L^2=0,$

    \item[(b)]$q_{-1}\beta+q_{-2}L^2=-p,$

    \item[(c)] $ p_0\beta+e^{\sigma} p_{-1}L^2=q,$

    \item[(d)]$ p_{-1}\beta+p_{-2}L^2=0,$

     \item[(e)]$q\beta+e^{\sigma} pL^2=f^2.$
\end{description}
  \end{lem}

\begin{lem}\label{yk}The following identities hold:\vspace{-0.2cm}
\begin{description}
   \item[(a)]$\dot{\partial}_iq=p_0m_i+q/L l_i, $

    \item[(b)]$\dot{\partial}_ip=p_{-1}m_i$,

    \item[(c)]$\dot{\partial}_ip_{0}=p_{02}m_i, $

    \item[(d)]$\dot{\partial}_ip_{-1}=-e^{-\sigma}(\beta/L^2)p_{02}m_i-(p_{-1}/L)l_i,$

\item[(e)]$\dot{\partial}_ip_{-2}=[e^{-\sigma}(\beta^2/L^4)p_{02}-
    (p_{-1}/L^2)]m_i+p_{-1}(2\beta/L^3)l_i.$
\end{description}
\end{lem}
\begin{lem}\label{sh1.2}
The following identities hold:
\begin{description}
\item[(a)]$\partial_k q=p_0N^r_km_r+qN^r_kl_r/L+p_0b_{0\mid
    k}+e^{\sigma}L^2p_{-1}\sigma_k$,

    \item[(b)] $\partial_kp=p_{-1}N^r_km_r+p_{-1}b_{0\mid k}+(p-\beta p_{-1})\sigma_k,$

    \item[(c)] $\partial_k p_{0}=p_{02}(N^r_km_r+b_{0\mid k}-\beta  \sigma_k),$

    \item[(d)] $\partial_kp_{-1}=-(p_{-1}/L)N^r_kl_r-e^{-\sigma}(\beta/L^2)(p_{02}N^r_km_r
        +p_{02}b_{0\mid k})+e^{-\sigma}(\beta^2/L^2)p_{02}\sigma_k,$

     \item[(e)]\,$\partial_kp_{-2}=[
    e^{-\sigma}(\beta^2/L^4)p_{02}-(p_{-1}/L^2)]N^r_km_r
    + (2\beta p_{-1}/L^3)N^r_kl_r\\{\quad\quad\quad}+[e^{-\sigma}(\beta^2/L^4)p_{02}
    -(p_{-1}/L^2)]b_{0\mid k}-
     e^{-\sigma}(\beta^3/L^4)p_{02} \,\sigma_k$.
\end{description}
  \end{lem}
Now, using Lemma \ref{f}, we get
\begin{prop}\label{h-g} Under a generalized $\beta$-conformal change,
we have:\vspace{-0.2cm}
\begin{description}
    \item[(a)]
 ${\quad}\overline{l}_i=e^\sigma f_1 l_i+f_2 b_i$,

    \item[(b)]
${\quad}\overline{h}_{ij}=e^\sigma p \,h_{ij}+q_0 m_i m_j$,

    \item[(c)]
        ${\quad}\overline{g}_{ij}=e^\sigma p\,g_{ij}+p_0\,b_ib_j
        +e^\sigma p_{-1}(b_iy_j+b_jy_i)+
        e^\sigma p_{-2}\,y_i\,y_j.$
\end{description}
\end{prop}

\begin{proof} As an illustration, we prove \textbf{(c)} only.
\begin{eqnarray*}
       \overline{g}_{ij}&=&
       \dot{\partial}_j\dot{\partial}_i(\frac{1}{2}\overline{L}^2)
        =\dot{\partial}_i(f(e^\sigma f_1 l_{j}+f_{2} b_j))\\
        &=&(e^\sigma f_1 l_{i}+f_{2} b_i)(e^\sigma f_1 l_{j}+f_{2} b_j)
        +f(e^\sigma f_1 l_{ij}+e^\sigma (e^\sigma f_{11}l_i+f_{12}b_i)l_j)\\
        &&+f(e^\sigma f_{12}l_i+f_{22}b_i)b_j\\
        & =&e^\sigma p \,g_{ij}+e^\sigma[e^{\sigma}(f^2_{1}/L^2)
        + e^{\sigma}(ff_{11}/L^2)-(ff_{1}/L^3)]y_iy_j+(f^2+q_0)b_ib_j\\
        &&+e^\sigma[(f_1f_2/L)+ (ff_{12}/L)(b_iy_j+b_jy_i)]\\
        &=&e^{\sigma}p\,g_{ij}+p_0b_ib_j+e^{\sigma}p_{-1}(b_iy_j+b_jy_i)+e^{\sigma}p_{-2}
        y_iy_j.
\end{eqnarray*}
\end{proof}

Lemma \ref{relations} helps us to compute the inverse metric
$\overline{g}^{ij}$ of the metric
 $\overline{g}_{ij}$.

\begin{prop}\label{inverse} Under a generalized $\beta$-conformal
change,
 the inverse metric $\overline{g}^{ij}$ of the metric
 $\overline{g}_{ij}$ is given by:
    \begin{eqnarray*}
\overline{g}^{ij}
&=&(e^{-\sigma}/p){g}^{ij}-s_0b^ib^j-s_{-1}(y^ib^j+y^jb^i)-s_{-2}y^iy^j,
\end{eqnarray*}
where \vspace{-0.2cm}
 \begin{eqnarray*}
s_0:&=&e^{-\sigma}f^2q_0/(\varepsilon pL^2),
\hspace{.2cm}s_{-1}:=p_{-1}f^2/(p \,\varepsilon L^2),\hspace{.2cm}
s_{-2}:=p_{-1}(e^{\sigma}m^2 p L^2-b^2f^2)/(\varepsilon p \beta
L^2),
\end{eqnarray*}
   $\varepsilon :=f^2(e^\sigma p+m^2 q_0)/L^2\neq 0$,\quad
 $m^2=g^{ij}m_im_j=m^im_i$.
  \end{prop}

  \begin{rem}\label{s0}
The quantities $s_{0}\,,\,s_{-1}$ and $s_{-2}$
 satisfy:
$$\beta s_0+L^2s_{-1}=q/\varepsilon,$$
$$b^2s_{-1}+\beta s_{-2}=e^{\sigma}p_{-1}m^2/\varepsilon.$$
\end{rem}

%%%%%%%%%%%%%%%%%%%%%%%%%%%%%%%%%%%%%%%%%%%%%%%%%%%%%%%%%%%%%%%%%%%%%%%%%%%%%%%%

\begin{prop}\label{c} Under a generalized $\beta$-conformal change, we
have\emph{:}
\begin{description}
    \item[(a)] The Cartan  tensor $\overline{C}_{ijk}$ is
    expressed in terms of $C_{ijk}$ as
    \begin{equation}\label{cc}
    \overline{C}_{ijk}=e^{\sigma} p\,C_{ijk}+V_{ijk},
\end{equation}

    \item[(b)] The (h)hv-torsion tensor $\overline{C}^{l}_{ij}$ is
    expressed in terms of $C^l_{ij}$ as
\begin{equation}\label{sh1-10}
   \overline{C}^l_{ij}=C^l_{ij}+M^l_{ij}\,,
\end{equation}
 where
\vspace{-.5 cm}\begin{eqnarray*}
    V_{ijk}&:=&\frac{e^{\sigma}p_{-1}}{2}(h_{ij}m_k+h_{jk}m_i+h_{ki}m_j)
    +\frac{p_{02}}{2}m_im_jm_k\\
and\qquad\qquad    M^l_{ij}&:=&\frac{1}{2p}[e^{-\sigma}m^l-p\,m^2
    (s_0b^l+s_{-1}y^l)](e^{\sigma}p_{-1}h_{ij}+p_{02}m_im_j)\\
    &&-e^{\sigma}(s_0b^l+s_{-1}y^l)(p\,C_{isj}b^s+p_{-1}m_im_j)\\
   &&+\frac{p_{-1}}{2p}(h^l_im_j+h^l_jm_i);\\
       h^i_j&=&g^{il}h_{lj}.
\end{eqnarray*}
\end{description}
\end{prop}

\begin{proof}We prove \textbf{(a)} only.
 Differentiating  $\overline{g}_{ij}$ (Proposition \ref{h-g}) with respect to
  $y^k$, using    Lemma
 \ref{yk},
we have
\begin{eqnarray*}
     2\overline{C}_{ijk}&=&2e^{\sigma}p \,C_{ijk}+
     e^{\sigma}g_{ij}p_{-1}m_k+p_{02}b_ib_jm_k+e^{\sigma}p_{-1}(b_ig_{jk}+b_jg_{ik})\\&&-e^{\sigma}(b_iy_j+b_jy_i)
    .[e^{-\sigma}(\beta/L^2)p_{02}m_k+(p_{-1}/L)l_k]+e^{\sigma}p_{-2}g_{ik}y_j\\&&+e^{\sigma}p_{-2}y_ig_{jk}
    +e^{\sigma}y_iy_j[(e^{-\sigma}(\beta^2/L^4)p_{02}-(p_{-1}/L^2))m_k+(2\beta/L^3)l_k]\\
    &=&2e^{\sigma}p\,C_{ijk}+e^{\sigma}p_{-1}(h_{ij}m_k+h_{jk}b_i+h_{ik}b_j)+p_{02}m_im_jm_k\\
    &&-e^{\sigma}p_{-1}(\beta/L)(h_{ik}l_j+h_{jk}l_i+2l_il_jl_k)+2e^{\sigma}(\beta/L)p_{-1}l_il_jl_k\\
    &=&2e^{\sigma}p\,C_{ijk}+e^{\sigma}p_{-1}(h_{ij}m_k+h_{jk}m_i+h_{ki}m_j)+p_{02}m_im_jm_k.
\end{eqnarray*}
\end{proof}

%%%%%%%%%%%%%%%%%%%%%%%%%%%%%%%%%%%%%%%%%%%%%%%%%%%%%%%%%%%%%%%%%%%%%%%%%%%
%%%%%%%%%%%%%%%%%%%%%%%%%

   The  transformed Christoffel symbols of the Finsler space $\overline{F}^n$  are given
   by$$\overline{\gamma}^i_{jk}=\frac{1}{2}\overline{g}^{ir}(\partial_j\overline{g}_{kr}
   +\partial_k\overline{g}_{rj}-\partial_r\overline{g}_{jk}).$$
In view of Lemma \ref{sh1.2} and the above expression, we get
 \begin{prop}\label{gamma} Under
a generalized $\beta$-conformal change, the Christoffel symbols
$\gamma^i_{jk}$ transform as follows:
\begin{eqnarray}
 \nonumber
   \overline{\gamma}^i_{jk}&=& \gamma^i_{jk}+\overline{g}^{ir}\left[
      F_{rk}Q_j+F_{rj}Q_k+E_{jk}Q_r-
      \Theta_{(j,k,r)}\set{B_{jk}b_{0\mid
      r}+V_{jkt}N^t_r+(1/2)K_{jk}\sigma_{r}}\right] \\
  &&+(g^{im}
      -e^{\sigma}p\overline{g}^{im})\Theta_{(j,k,m)}\{C_{jkr}N^r_m\},
\end{eqnarray}
\end{prop}
\hspace{-0.7cm}where
\begin{eqnarray*}
   2E_{ij} &=& b_{i\mid j}+b_{j\mid i},\,\,2F_{ij}=b_{i\mid j}-b_{j\mid i}, \\
    2B_{ij}&=&e^\sigma p_{-1}h_{ij}+p_{02}m_im_j,  \\
    Q_i&=& e^\sigma p_{-1}y_i+p_0 b_i,\\
      K_{ij}&=&A_1g_{ij}+A_2b_ib_j+A_3(b_iy_j+b_jy_i)+A_4y_iy_j,\\
     A_1 &=&e^\sigma(2p-\beta p_{-1}), A_2=-\beta p_{02}, A_3=e^\sigma
p_{-1}+(\beta^2/L^2)p_{02}\,,\,\,A_4=e^\sigma
p_{-2}-(\beta^3/L^4)p_{02},
\end{eqnarray*}

\begin{proof}  After long but easy calculations, using Lemma \ref{sh1.2}, one can show that
\begin{eqnarray*}
     \partial_j\overline{g}_{kr}
     &=&e^\sigma p\,
 \partial_jg_{kr}+2N^l_jV_{krl}+2B_{kr}b_{0\mid j}+Q_rb_{k\mid j}+Q_kb_{r\mid
 j}+p_0(b_rb_l\Gamma^l_{kj}+b_lb_k\Gamma^l_{rj})\\
&+&e^\sigma
p_{-1}((b_ry_l+b_ly_r)\Gamma^l_{kj}+(b_ly_k+b_ky_l)\Gamma^l_{rj})+e^\sigma
p_{-2}(y_ry_l\Gamma^l_{kj}+y_ky_l\Gamma^l_{rj}) +K_{kr}\sigma_j.
\end{eqnarray*}
 The result follows from the above formula and the definition of $\overline{\gamma}^i_{jk}$.
 \end{proof}
%%%%%%%%%%%%%%%%%%%%%%%%%%%%%%%%%%%%%%%%%%%%%%%%%%%%%%%%%%%%%%%%%%%%%%%%%%
Propositions \ref{h-g}, \ref{inverse}, \ref{c} and \ref{gamma}
constitute the main elementary entities, or  building blocks, of
the geometry of the transformed space $\overline{F}^n$. As a
result, we are now in a position to construct the fundamental
geometric objects of such geometry.

 Firstly,  the following result
determines the change of the canonical spray and Cartan nonlinear
connection under a generalized $\beta$-conformal change.
\newpage
\begin{thm}\label{di}
Under a generalized $\beta$-conformal change, we have:
\begin{description}
\item[(a)] The change of the canonical spray $G^i$ is given by
$$\overline{G}^i=G^i+D^i,$$
 where
\begin{eqnarray}
 \label{eqn.2}  \nonumber D^i&=&\frac{\sigma_0}{2p}\{[2p-\beta
p_{-1}-e^{\sigma}p^2L^2s_{-2} -ps_{-1}(2 e^{\sigma}
p\beta+e^{\sigma}p_{-1}L^2m^2)]y^i -2e^\sigma p^2\beta s_0b^i\}\\
 &&+ \frac{q}{p} e^{-\sigma} F^i_0-\frac{1}{2}L^2\sigma^i+\frac{1}{2}(e^{\sigma}pE_{00}
  -2qF_{\beta0}+e^{\sigma}pL^2\sigma_\beta)(s_0b^i+s_{-1}y^i);\\
\nonumber  F^i_0&=&F_{jk}y^kg^{ij},\, F_{\beta 0}=F_{r0}b^r,\,
\sigma_\beta=\sigma_rb^r.
\end{eqnarray}
 \item[(b)]The change of the Cartan nonlinear  connection $N^i_j$
 is given by

$$\overline{N}^i_j=N^i_j+D^i_j, $$
where
\begin{eqnarray}
\label{eqn.3}
\nonumber D^i_j&=&\frac{e^{-\sigma}}{p}A^i_j-(s_0b^i+s_{-1}y^i)A_{rj}b^r\\
 &&-(qb_{0\mid j}+e^\sigma p L^2\sigma_j)(s_{-1}b^i+s_{-2}y^i);\\
\nonumber A_{ij}&:=&E_{00}B_{ij}+F_{i0}Q_j+qF_{ij}+E_{j0}Q_i-2(e^\sigma p\,C_{sij}+V_{sij})D^s\\
  \nonumber    &&+\frac{1}{2}\sigma_0[2e^\sigma pg_{ij}+2e^\sigma p_{-1}m_jy_i-2\beta
     B_{ij}+e^{\sigma}p_{-1}(b_iy_j-b_jy_i)]\\
 \nonumber     &&-\frac{1}{2}\sigma_i(e^\sigma L^2p_{-1}m_j+2e^\sigma py_j)
    +\frac{1}{2}\sigma_j(2e^\sigma py_i+e^{\sigma}L^2 p_{-1}m_i),\\
   \nonumber A^i_j&=&g^{li}A_{lj}.
\end{eqnarray}
\end{description}
\end{thm}
\begin{proof}~\par

    \textbf{(a)} Using proposition \ref{gamma} and the expression
    $\overline{G}^i=\frac{1}{2}\overline{\gamma}^i_{jk}y^jy^k$,  we
get
\begin{eqnarray}
\label{salah}
\overline{G}^i&=&G^i+\frac{1}{2}\overline{g}^{ir}[2qF_{r0}+E_{00}Q_r-
e^\sigma pL^2\sigma_r+\sigma_0(2e^\sigma
py_r+e^\sigma L^2 p_{-1}m_{r})]\\
   \nonumber &=&G^i+\frac{q}{p}  e^{-\sigma} F^i_0-\frac{1}{2}L^2\sigma^i+\frac{1}{2}(e^{\sigma}pE_{00}
   -2qF_{\beta0}+e^{\sigma}pL^2\sigma_\beta)(s_0b^i+s_{-1}y^i)\\
   \nonumber &&-\frac{\sigma_0}{2p}\{2e^\sigma p^2\beta s_0b^i-[2p-
p_{-1}\beta-e^{\sigma}p^2L^2s_{-2} -p s_{-1}(2 e^{\sigma}
p\beta+e^{\sigma}p_{-1}L^2m^2)]y^i\}.
\end{eqnarray}

    \textbf{(b)} Differentiating $D^i$ with respect to  $y^j$, we have
\begin{eqnarray*}
   D^i_j&:=& \dot{\partial}_jD^i=\frac{1}{2}\dot{\partial}_j\{\overline{g}^{ir}
   [2qF_{r0}+E_{00}Q_r-e^\sigma pL^2\sigma_r+\sigma_0(2e^\sigma
py_r+e^\sigma L^2p_{-1}m_{r})]\}\\
   &=&\overline{g}^{ir}\{E_{00}B_{rj}+F_{r0}Q_j
    +qF_{rj}+E_{j0}Q_r-2(e^\sigma p\,C_{srj}+V_{srj})D^s \\
   &&+\frac{1}{2}\sigma_0(2e^\sigma pg_{jr}+2e^\sigma p_{-1}m_jy_r-2\beta
     B_{jr}+e^\sigma p_{-1}(b_ry_j-b_jy_r))\\
     &&-\frac{1}{2}\sigma_r(e^\sigma L^2p_{-1}m_j+2e^\sigma pLl_j)+\frac{1}{2}\sigma_j(2e^\sigma py_r
     +e^\sigma L^2p_{-1}m_{r})\}\\
     &=&\frac{e^{-\sigma}}{p}A^i_j-(s_0b^i+s_{-1}y^i)A_{rj}b^r-(qb_{0\mid
     j}+e^\sigma pL^2\sigma_j)(s_{-1}b^i+s_{-2}y^i).
\end{eqnarray*}
This ends the proof.
\end{proof}

As a direct consequence of the above theorem, the coefficients of
the Berwald  connection $B\overline{\Gamma}$ of the transformed
Finsler space $\overline{F}^n$ can be computed as follows.

\begin{thm}\label{new}Under a generalized  $\beta$-conformal
change,
the coefficients of the Berwald connection $\overline{G}^i_{jk}$ are
given by
     $$\overline{G}^i_{jk}
       =G^i_{jk}+B^i_{jk},$$
where $B^i_{jk}:=\dot{\partial}_kD^i_j$.
\end{thm}

Now, we are  in a position to announce one of the main results of
the present paper. Namely,

\begin{thm}\label{gammA_2}Under a generalized $\beta$-conformal change,
 the coefficients of the Cartan
connection $\overline{\Gamma}^i_{jk}$ are given by

 $$\overline{\Gamma}^i_{jk}
       =\Gamma^i_{jk}+D^i_{jk},$$

\end{thm}
 {\hspace{-.7cm}}where
 {\vspace{-.3cm}}\begin{eqnarray}
   \nonumber   D^i_{jk}&:=&[(e^{-\sigma}/p){g}^{ir}-(s_0b^i+s_{-1}y^i)b^r-(s_{-1}b^i+s_{-2}y^i)y^r][
      F_{rk}Q_j+F_{rj}Q_k+E_{jk}Q_r\\
  \label{eqn.4}    &&+\frac{1}{2}\Theta_{(j,k,r)}\{2e^\sigma p \,C_{jkm}
      D^m_r+2V_{jkm}D^m_r- K_{jk}\sigma_{r}-2B_{jk}b_{0\mid
      r}\}].
\end{eqnarray}

\begin{proof}To compute $\overline{\Gamma}^i_{jk}$, we use  Propositions \ref{inverse},  \ref{c},
\ref{gamma} and
 Theorem \ref{di}:
 \begin{eqnarray*}
 \overline{\Gamma}^i_{jk}&=&\overline{\gamma}^i_{jk}+\overline{g}^{im}
     (\overline{C}_{jkr}\overline{N}^r_{m}-
     \overline{C}_{mkr}\overline{N}^r_{j}-\overline{C}_{jmr}\overline{N}^r_{k})\\
   &=&  \gamma^i_{jk}+(g^{im}-e^{\sigma}p\overline{g}^{im})(C_{jkr}N^r_m-C_{krm}N^r_j-
      C_{jrm}N^r_k)+\overline{g}^{ir}(B_{jr}b_{0\mid k}\\
      &&+B_{kr}b_{0\mid j}-B_{jk}b_{0\mid r}+F_{rk}Q_j+F_{rj}Q_k+E_{jk}Q_r+
      N^t_jV_{krt}+N^t_kV_{jrt}
      -N^t_rV_{jkt})\\
     &&+\frac{1}{2}\overline{g}^{ir}(\sigma_{k}K_{jr}+\sigma_{j}K_{kr}-\sigma_{r}K_{jk})
     +\overline{g}^{im}\{(e^\sigma p
     \,C_{jkr}+V_{jkr})(N^r_m+D^r_m)\\
     &&-(e^\sigma p\, C_{krm}+V_{krm})(N^r_j+D^r_j)-(e^\sigma p\, C_{jmr}+V_{jmr})(N^r_k+D^r_k)\}\\
   &=&
   \gamma^i_{jk}+g^{im}(C_{jkr}N^r_m-C_{krm}N^r_j-C_{jrm}N^r_k)
   +\overline{g}^{ir}\{B_{jr}b_{0\mid k}+B_{kr}b_{0\mid j}-
      B_{jk}b_{0\mid r}\\
      &&+F_{rk}Q_j+F_{rj}Q_k+E_{jk}Q_r+\frac{1}{2}(\sigma_{k}K_{jr}
      +\sigma_{j}K_{kr}-\sigma_{r}K_{jk})\\
      &&+e^\sigma p \,C_{jkm}
      D^m_r+V_{jkm}D^m_r-e^\sigma p\,C_{rkm}D^m_j-V_{rkm}D^m_j-e^\sigma
      p\,C_{rjm}D^m_k-V_{rjm}D^m_k\}\\
      &=&\Gamma^i_{jk}+[(e^{-\sigma}/p){g}^{ir}-(s_0b^i+s_{-1}y^i)b^r
      -(s_{-1}b^i+s_{-2}y^i)y^r][
      F_{rk}Q_j+F_{rj}Q_k+E_{jk}Q_r\\
      &&+ \frac{1}{2}\Theta_{(j,k,r)}\{2e^\sigma p \,C_{jkm}
      D^m_r+2V_{jkm}D^m_r-K_{jk}\sigma_{r}-2B_{jk}b_{0\mid
      r}\}].
\end{eqnarray*}
 This completes the proof.
\end{proof}
\begin{cor} The tensor $D^i_{jk}$  has
the  properties:
\begin{equation}\label{sh1-19}
D^i_{j0}=B^i_{j0}=D^i_j,\qquad D^i_{00}=2D^i.
\end{equation}
\end{cor}
In what follows we say that $A_i$, for example,  is Cartan-parallel
to mean that $A_i$ is parallel with respect to the \emph{horizontal}
covariant derivative of Cartan connection: $A_{i|j}=0.$ Similarly,
for the other connections  existing in the space.
\begin{thm} \label{abd}
Under a generalized $\beta$-conformal change
$L\rightarrow\overline{L}=f(e^\sigma L,\beta)$, consider the
following two assertions:\begin{description}
    \item[(i)] The covariant vector $b_i$ is Cartan-parallel.

    \item[(ii)]The difference tensor $D^i_{jk}$ vanishes
    identically.

\item[]Then, we have:

    \item[(a)] If \textbf{(i)} and \textbf{(ii)} hold, then $\sigma$ is homothetic.

    \item[(b)] If $\sigma$ is homothetic, then \textbf{(i)} and \textbf{(ii)} are
    equivalent.
\end{description}
\end{thm}
\begin{proof}~\par
  \noindent \textbf{(a)} If $D^i_{jk}=0$, then, by (\ref{sh1-19}),
    $D^i=0$ (i.e, $\overline{G}^i=G^i$).
     Moreover,   $b_{j\mid k}=0 $  implies
    that  $ F_{jk}=E_{jk}=0$. Consequently,   (\ref{salah}) reduces to
$$ pL^2\sigma_r-\sigma_0(2 py_r+L^2p_{-1}m_{r})=0.$$
 Now, transvecting  with $y^r$,  we get
$\sigma_0=0$. From which, the above equation implies that
$\sigma_r=0$.  That is, $\sigma$ is homothetic.\\

 \noindent \textbf{(b)} Let   $\sigma$ be homothetic and  $b_{j\mid
    k}=0$. Then, $D^i=0$,  by (\ref{eqn.2}).  Consequently,
    $D^i_{jk}=0$ by (\ref{eqn.4}).

     On the other hand,  let $\sigma$ be homothetic and $D^i_{jk
    }=0$. Then, by (\ref{sh1-19}), $D^i=0$. Hence, (\ref{eqn.2}) reduces to
    \begin{equation}\label{help}
\frac{e^{-\sigma}q}{p}F^i_0+\frac{1}{2}(e^{\sigma}pE_{00}-2qF_{\beta0}
    )(s_0b^i+s_{-1}y^i)=0.
\end{equation}
    Transvecting  (\ref{help}) with $y_i$ and since
    $s_0\beta+s_{-1}L^2\neq 0$   (by Remark \ref{s0}), we get
    \begin{equation}\label{eqn.5}
e^\sigma p E_{00}-2qF_{\beta 0}=0.
\end{equation}  This,
    together with (\ref{help}),
    imply that    $F^i_{0}=0$. Consequently,   $E_{00}=0$,
    by (\ref{eqn.5}). Since
    $F_{ij}=\dot{\partial}_j F_{i0}$ and
    $E_{j0}=\dot{\partial}_j{E_{00}}$,
    then $F_{ij}=0$ and $E_{j0}=0$,
    which leads to $b_{i|0}=b_{0|i}=0$. Consequently,
    $0=D^i_{jk}=\overline{g}^{ir}E_{jk}Q_r$, by (\ref{eqn.4}).
    Hence, $E_{jk}Q_r=0$ and transvecting this with  $y^r$
     gives $E_{jk}=0$. Then, the result follows from the definition of $F_{jk}$ and $E_{jk}$.
\end{proof}

As a consequence of the above theorem, we have the following
interesting special cases.
\begin{thm} \label{cases}~\par
    \begin{description}
        \item[(a)]Let the generalized  $\beta$-conformal change $L\rightarrow\overline{L}=
        f(e^\sigma
        L,\beta)$ be a conformal change ($\beta=0$), then
        $D^i_{jk}$
        vanishes identically if and only if $\sigma$ is
        homothetic.

        \item[(b)]Let the generalized $\beta$-conformal change $L\rightarrow\overline{L}=
        f(e^\sigma
        L,\beta)$ be a $\beta$-change ($\sigma=~0$), then
        $D^i_{ij}$ vanishes identically if and only if $b_i$ is
        Cartan-parallel.
    \end{description}
\end{thm}

A Finsler space $F^n=(M,L)$ is called a Berwald space if the
Berwald connection coefficients  $G^i_{jk}$ are function of the
positional argument  $x^i$ only. As an immediate consequence of
Theorems \ref{abd} and \ref{cases},
  we have

\begin{thm} \label{berwald}
   Consider a generalized $\beta$-conformal change $L\rightarrow \overline{L}$ having the properties
    that  $b_i(x)$  is Cartan-parallel  and
    $\sigma$   is homothetic. If the original  space  $F^n$ is a Berwald space,
    then so is the transformed space $\overline{F}^n$.
\end{thm}

\begin{cor}\label{cor1} Let the Finsler structure $L$ on $F^n$ be
Riemannian. Assume that  $b_i(x)$ is Riemann-parallel and $\sigma$
is homothetic. Then, the transformed space $\overline{F}^n$ is a
Berwald space.
\end{cor}

It is to be noted  that Theorem \ref{abd}, Theorem   \ref{berwald}
and Corollary \ref{cor1}   generalize some of Shibata's
results \cite{r1.13} and Abed's results \cite{r1.2}.\\

We conclude   this section with the following  result which
determines the  coefficients of the fundamental linear connections
in  Finsler geometry.\\

\begin{thm}\label{final.0} Under the  generalized $\beta$-conformal change (\ref{change}),
\begin{description}
    \item[(a)] the  transformed Cartan connection has the form
    $\overline{C\Gamma}=( \overline{\Gamma}^{h}_{ij}, \overline{N}^{h}_{i}, \overline{C}^{h}_{ij}),$
    \item[(b)]the transformed  Chern connection has the form
    $\overline{R\Gamma}=( \overline{\Gamma}^{h}_{ij}, \overline{N}^{h}_{i}, 0),$
    \item[(c)]the transformed Hashiguchi  connection has the form
    $\overline{H\Gamma}=( \overline{G}^{h}_{ij}, \overline{N}^{h}_{i}, \overline{C}^{h}_{ij}),$
    \item[(d)]the transformed Berwald connection has the form
    $\overline{B\Gamma}=( \overline{G}^{h}_{ij}, \overline{N}^{h}_{i},0),$
\end{description}
where the coefficients $\overline{N}^{h}_{i}$,
$\overline{G}^{h}_{ij}$ and $\overline{\Gamma}^{h}_{ij}$  are
given by   Theorem \ref{di}, Theorem \ref{new}  and Theorem
\ref{gammA_2} respectively, whereas the  components   $
\overline{C}^{h}_{ij}$ are  given by Proposition \ref{c}.
\end{thm}

%%%%%%%%%%%%%%%%%%%%%%%%%%%%%%%%%%%%%%%%%%%%%%%%%%%%%%%%%
%%%%%%%%%%%%%%%%%%%%% Section . %%%%%%%%%%%%%%%%%%%%%%%

\Section{\textbf{Change of the torsion and  curvature tensors}} In
this section, we  consider how the torsion and curvature tensors
transform under  the generalized $\beta$-conformal change
(\ref{change}).%\\

 For an arbitrary Finsler connection
$F\Gamma=(\textbf{F}^i_{jk},\textbf{N}^i_j,\textbf{C}^i_{jk})$ on
the space $F^n$, the  (h)h-,  (h)hv-,  (v)h-,  (v)hv- and
(v)v-torsion tensors of $F\Gamma$ are respectively given
by~\cite{r1.6}:
\begin{eqnarray*}
 \textbf{T}^i_{jk}&=&\textbf{F}^i_{jk}-\textbf{F}^i_{kj},\\
 \textbf{C}^i_{jk}&=& \text{the connection parameters $\textbf{C}^i_{jk}$} ,\\
   \textbf{R}^i_{jk}&=&\delta_k\textbf{N}^i_j-\delta_j\textbf{N}^i_k,\\
    \textbf{P}^i_{jk}&=&\dot{\partial}_k\textbf{N}^i_j-\textbf{F}^i_{jk},\\
    \textbf{S}^i_{jk}&=&\textbf{C}^i_{jk}-\textbf{C}^i_{kj}.
\end{eqnarray*}
The h-,  hv- and v-curvature  tensors of $F\Gamma$ are respectively
given by \cite{r1.6}:
\begin{eqnarray*}
  \textbf{R}^i_{hjk} &=& \mathfrak{A}_{(j,k)}\{{\delta_k\textbf{F}^i_{hj}}
+\textbf{F}^m_{hj}\textbf{F}^i_{mk}\}+\textbf{C}^i_{hm}\textbf{R}^m_{jk}, \\
  \textbf{ P}^i_{hjk}&=& \dot{\partial}_k\textbf{F}^i_{hj}-\textbf{C}^i_{hk\mid
j}+\textbf{C}^i_{hm}\textbf{P}^m_{jk},\\
  \textbf{S}^i_{hjk}&=&\mathfrak{A}_{(j,k)}\{\dot{\partial}_k\textbf{C}^i_{hj}+\textbf{C}^m_{hk}\textbf{C}^i_{mj}\}.
\end{eqnarray*}

\vspace{.1cm} The next table provides a comparison concerning  the
four fundamental linear connections and their associated torsion
and curvature tensors. The explicit expressions of such tensors,
under a generalized $\beta$-conformal change, will be given just
after the table. It should be noted that the geometric objects
associated with Chern connection, Hashiguchi connection and
Berwald connection will be marked by $\star$, $\ast$ and $\circ$
respectively. For Cartan connection no special symbol is assigned.

\begin{center}{\bf{Table 1: Fundamental linear connections}}
\\[0.5cm]
\small{
\begin{tabular}
{|c|c|c|c|c|c|}\hline
 &\multirow{2}{*}{\textbf{} }&\multirow{2}{*}{{\bf Cartan }}&
 \multirow{2}{*}{ {\bf Chern  }} &\multirow{2}{*}{{\bf Hashiguchi }}&\multirow{2}{*}{{\bf
 Berwald}}\\
 &&&&&\\[0.1 cm]\cline{2-6}
&\multirow{2}{*}{$(\textbf{F}^{h}_{ij},
\textbf{N}^{h}_{i},\textbf{C}^{h}_{ij})$}&
\multirow{2}{*}{$(\Gamma^{h}_{ij}, N^{h}_{i}, C^{h}_{ij})$}&
\multirow{2}{*}{$( \Gamma^{h}_{ij}, N^{h}_{i}, 0)$}&
\multirow{2}{*}{$( G^{h}_{ij}, N^{h}_{i}, C^{h}_{ij})$}&
\multirow{2}{*}{$( G^{h}_{ij}, N^{h}_{i}, 0)$}\\
\rule{.5cm}{0pt}
\begin{rotate}{90}\hspace{-.15cm}Connections\end{rotate}\rule{.5cm}{0pt}&&
& & &
\\[0.1 cm]\hline
&\multirow{2}{*}{{\bf (h)h-tors.} $\textbf{T}^{i}_{jk}$}&\multirow{2}{*}{$0$}&\multirow{2}{*}{$0$}&\multirow{2}{*}{$0$}&\multirow{2}{*}{$0$}\\
&\multirow{2}{*}{{\bf (h)hv-tors.}
$\textbf{C}^{i}_{jk}$}&\multirow{2}{*}{$C^{i}_{jk}$} &
\multirow{2}{*}{$0$}&
\multirow{2}{*}{$C^{i}_{jk}$}&\multirow{2}{*}{$0$}
\\[0.1 cm]&&&&&
\\[0.1 cm]\cline{2-6}
&&&&&\\
\rule{.5cm}{0pt}
\begin{rotate}{90}Torsions\end{rotate}\rule{.5cm}{0pt}
&{\bf (v)h-tors.} $\textbf{R}^{i}_{jk}$&${R}^{i}_{jk}$
&$\overdiamond{R}^{i}_{jk}={R}^{i}_{jk}$&
$\overast{R}^{i}_{jk}=R^{i}_{jk}$&
$\overcirc{R}^{i}_{jk}={R}^{i}_{jk}$\\[0.1 cm]
&{\bf (v)hv-tors.}
$\textbf{P}^{i}_{jk}$&${P}^{i}_{jk}=C^{i}_{jk|0}$&
  $\overdiamond{P}^{i}_{jk}=P^{i}_{jk}$&$0$ &$0$
\\[0.1 cm]
&{\bf (v)v-tors.} $\textbf{S}^{i}_{jk}$&$0$ &$0$&$0$ &$0$
\\[0.1 cm]\hline
&&&&&\\
&{\bf h-curv.} $\textbf{R}^{h}_{ijk}$& ${R}^{h}_{ijk}$&
$\overdiamond{R}^{h}_{ijk}$&
$\overast{R}^{h}_{ijk}$&$\overcirc{R}^{h}_{ijk}$
\\[0.1 cm]
&{\bf hv-curv.} $\textbf{P}^{h}_{ijk}$& ${P}^{h}_{ijk}$&
$\overdiamond{P}^{h}_{ijk}$&
$\overast{P}^{h}_{ijk}$&$\overcirc{P}^{h}_{ijk}$
\\[0.1 cm]
\rule{.5cm}{0pt}
\begin{rotate}{90}\hspace{.2cm}Curvatures\end{rotate}\rule{.5cm}{0pt}&{\bf v-curv.}
$\textbf{S}^{h}_{ijk}$& ${S}^{h}_{ijk}$&
$0$&$\overast{S}^{h}_{ijk}=S^{h}_{ijk}$ &$0$
\\[0.1 cm]\hline &&&&&\\
&{\bf h-cov. der.}&
${K}^{i}_{j|k}$&${K}^{i}_{j\stackrel{\star}|k}={K}^{i}_{j|k}$ &
${K}^{i}_{j\stackrel{*}|k}$&${K}^{i}_{j\stackrel{\circ}|k}={K}^{i}_{j\stackrel{*}|k}$
\\[0.1 cm]
&{\bf v-cov. der.}&
${K}^{i}_{j}|_k$&${K}^{i}_{j}\stackrel{\star}|_k=\paa_{k}{K}^{i}_{j}$
&${K}^{i}_{j}\stackrel{*}|_k={K}^{i}_{j}|_k$
&${K}^{i}_{j}\stackrel{\circ}|_k=\paa_{k}{K}^{i}_{j}$
\\
\rule{.2cm}{0pt}
\begin{rotate}{90}\hspace{.2cm} Covariant \end{rotate}\rule{.5cm}{0pt}
\rule{0cm}{0pt}\begin{rotate}{90}
\hspace{.3cm}derivatives\end{rotate}\rule{.5cm}{0pt}&&&&&\\ \hline
\end{tabular}}

\end{center}
\par
It is to be noted that the explicit expressions of the geometric
objects of the above table can be found in \cite{r1.6}.

% \vspace{7pt}
\par
Now,  by Theorem \ref{final.0}, one  can prove the following
\begin{thm}\label{final}
Under  a generalized $\beta$-conformal change, the torsion and
curvature tensors of Cartan, Chern, Hashiguchi and Berwald
connections, are given by:
\begin{description}
\item[(a)]For Cartan connection, we have
\begin{eqnarray*}
 \overline{C}^i_{jk} &=&C^i_{jk}+M^i_{jk},  \\
   \overline{R}^i_{jk}&=&  R^i_{jk}+\mathfrak{A}_{(j,k)}\{D^i_{j\mid
    k}-(B^i_{jr}+P^i_{jr})D^r_k\},\\
   \overline{P}^i_{jk}&=& P^i_{jk}-D^i_{jk}+B^i_{jk},\\
   \overline{R}^i_{hjk}&=&R^i_{hjk}+2S^i_{hrt}D^r_jD^t_k+M^i_{ht}R^t_{jk}
    -\mathfrak{A}_{(j,k)}\{A^i_{hj\mid k}-A^i_{hjt}D^t_k \\
    &&+D^i_{tj}D^t_{hk}+P^i_{hjt}D^t_k+M^i_{rh}P^r_{jt}D^t_k-M^i_{th}D^t_{j\mid k}
   +M^i_{ht}B^t_{jr}D^r_k\},\\
  \overline{P}^i_{hjk} &=& P^i_{hjk}-2S^i_{thk}D^t_j-A^i_{hjk}+C^i_{tk}D^t_{hj}-C^t_{hk}D^i_{tj}
    +M^i_{th}P^t_{jk}
    +A^i_{jt}M^t_{hk}\\&&-M^i_{tk}A^t_{hj}-M^i_{hk\mid j}+M^i_{th}B^t_{jk}
    +M^i_{tkh}D^t_j
  +C^i_{rh}M^r_{kt}D^t_j -M^i_{hr}C^r_{tk}D^t_j,\\
 \overline{S}^i_{hjk} &=&S^i_{hjk}+\mathfrak{A}_{(j,k)}\{
    C^t_{hk}M^i_{tj}-C^i_{tk}M^t_{hj}-M^i_{tk}M^t_{hj}\},
\end{eqnarray*}

where $ A^i_{jk}=-D^i_{jk}-C^i_{jt}D^t_k,\quad
A^i_{jkh}=\dot{\partial}_hA^i_{jk}$ and
$M^i_{jkh}=\dot{\partial}_hM^i_{jk}.$
 \item[(b)] For Chern Connection, we have
\begin{description}
    \item[]$\overlind{R}^i_{jk}=R^i_{jk}+\mathfrak{A}_{(j,k)}\{D^i_{j\mid
    k}-(B^i_{jr}+P^i_{jr})D^r_k\},$

    \item[]$\overlind{P}^i_{jk}=P^{i}_{jk}-D^i_{jk}+B^i_{jk},$

    \item[]$\overlind{R}^i_{hjk}=\overdiamond{R}^{ i}_{hjk}
    +\mathfrak{A}_{(j,k)}\{D^i_{hj\mid k}-D^i_{hjt}D^t_k
    -D^i_{tj}D^t_{hk}
   -\overdiamond{P}^{ i}_{hjt}D^t_k\},$

    \item[]$\overlind{P}^i_{hjk}=\overdiamond{P}^{i}_{hjk}
    +D^i_{hjk}$,
\end{description}
where $D^i_{jkh}=\dot{\partial}_hD^i_{jk}.$
 \item[(c)]For  Hashiguchi connection, we have
 \begin{eqnarray*}
   \overlina{C}^i_{jk} &=&C^i_{jk}+M^i_{jk},  \\
   \overlina{R}^i_{jk} &=&R^i_{jk}+\mathfrak{A}_{(j,k)}\{D^i_{j{\stackrel{*}\mid}
    k}-B^i_{jr}D^r_k\}, \\
  \overlina{R}^i_{hjk} &=&\overast{R}^i_{hjk}+2 \,\overast{S}^i_{hrt}D^r_jD^t_k+M^i_{ht}\overast{R}^t_{jk}
    -\mathfrak{A}_{(j,k)}\{H^i_{hj{\stackrel{*}\mid} k}-(\dot{\partial}_t H^i_{hj})D^t_k \\
   && +D^i_{tj}D^t_{hk}
   +\overast{P}^i_{hjt}D^t_k-M^i_{th}D^t_{j{\stackrel{*}\mid} k}
   +M^i_{ht}B^t_{jr}D^r_k\}, \\
   \overlina{P}^i_{hjk}&=& \overast{P}^i_{hjk}-2 \,\overast{S}^i_{thk}D^t_j-\dot{\partial}_kH^i_{hj}
    +C^i_{tk}D^t_{hj}-C^t_{hk}D^i_{tj}
    +H^i_{jt}M^t_{hk}-M^i_{tk}H^t_{hj}-M^i_{hk{\stackrel{*}\mid} j} \\
   &&+M^i_{th}B^t_{jk}
   +M^i_{kth}D^t_j
   +C^i_{rh}M^r_{kt}D^t_j-M^i_{hr}C^r_{tk}D^t_j \\
\overlina{S}^i_{hjk}
&=&S^i_{hjk}+\mathfrak{A}_{(j,k)}\{C^t_{hk}M^i_{tj}-C^i_{tk}M^t_{hj}
    -M^i_{tk}M^t_{hj}\},
 \end{eqnarray*}

where $H^i_{jk}=-B^i_{jk}-C^i_{jt}D^t_k.$
 \item[(d)]For  Berwald connection, we have

\begin{description}
    \item[]$\overlinc{R}^i_{jk}=R^i_{jk}+\mathfrak{A}_{(j,k)}\{D^i_{j{\stackrel{\circ}\mid}
    k}-B^i_{jr}D^r_k\},$

    \item[]$\overlinc{R}^i_{hjk}=\overcirc{R}^i_{hjk}
    -\mathfrak{A}_{(j,k)}\{B^i_{hj{\stackrel{\circ}\mid} k}-B^i_{hjt}D^t_k+D^i_{tj}D^t_{hk}
   +\overcirc{P}^i_{hjt}D^t_k\},$

    \item[]$\overlinc{P}^i_{hjk}=\overcirc{P}^i_{hjk}+B^i_{hjk}$,
\end{description}
where $B^i_{jkh}=\dot{\partial}_hB^i_{jk}.$
\end{description}
\end{thm}

\begin{cor}\label{final-1}
   Under  a generalized $\beta$-conformal change, for which
    the covariant vector field $b_i$ is Cartan-parallel  and $\sigma$ is
homothetic, we have:
\begin{description}

\item[(a)]The torsion and
    curvature tensors of Chern  connection are invariant.

\item[(b)]The torsion and
    curvature tensors of  Berwald connection are invariant.

\item[(c)]For Cartan connection, $R^i_{jk}$ and $P^i_{jk}$ are
invariant and

\begin{description}
\item[]$\overline{C}^i_{jk}=C^i_{jk}+M^i_{jk}$,

    \item[]$\overline{R}^i_{hjk}=R^i_{hjk}+M^i_{ht}R^t_{jk},$

    \item[]$\overline{P}^i_{hjk}=P^i_{hjk}
    +M^i_{th}P^t_{jk}-M^i_{hk\mid j}$,

    \item[]$\overline{S}^i_{hjk}=S^i_{hjk}+\mathfrak{A}_{(j,k)}\{C^t_{hk}M^i_{tj}-C^i_{tk}M^t_{hj}
    -M^i_{tk}M^t_{hj}\}.$
\end{description}

\item[(d)]For  Hashiguchi connection, $\overast{R}^i_{jk}$ is
invariant and

\begin{description}
\item[]$\overlina{C}^i_{jk}=C^i_{jk}+M^i_{jk}$,

 \item[]$\overlina{R}^i_{hjk}=\overast{R}^i_{hjk}+M^i_{ht}\overast{R}^t_{jk},$

 \item[]$\overlina{P}^i_{hjk}=\overast{P}^i_{hjk}-M^i_{hk{\stackrel{*}\mid} j}$,

    \item[]$\overlina{S}^i_{hjk}=\overast{S}^i_{hjk}+\mathfrak{A}_{(j,k)}\{C^t_{hk}M^i_{tj}-C^i_{tk}M^t_{hj}
    -M^i_{tk}M^t_{hj}\}.$
\end{description}

\end{description}
\end{cor}
~\par
A Finsler  space $F^n$ is called a Landesberg space  if
 the hv-curvature tensor $P^{i}_{hjk}$ of $C\Gamma$
vanishes, or equivalently $P^i_{jk}=0.$

By Corollary  \ref{final-1},  $P^i_{jk}$ is invariant under a
generalized $\beta$-conformal change for which  $b_i$ is
Cartan-parallel  and $\sigma$ is homothetic. Hence, we have the
following

\begin{thm}\label{landesberg} A  Landesberg space remains    Landesberg  under  a
generalized $\beta$-conformal change  if $b_{i}$  is
Cartan-parallel and $\sigma$ is homothetic.
\end{thm}

By the above theorem and the fact that the hv-curvature tensor
$P^{i}_{hjk}$ of a Riemannian space vanishes identically, we have
\begin{cor}\label{rlandesberg}
Under a generalized $\beta$-conformal change, a Riemannian space
$F^n$ is transformed to a Landesberg space $\overline{F}^n$  if
  $b_{i}$ is Riemann-parallel  and $\sigma$ is homothetic.
\end{cor}

A Finsler space  $F^n$ is called locally Minkowskian  if $F^n$ is a
Berwald space and the h-curvature tensor $R^i_{hjk}$ vanishes.

\begin{thm}\label{minkowski.1}
Assume that  the covariant vector $b_i(x)$ is Cartan-parallel and
$\sigma$ is homothetic. If $F^n$ is locally Minkowskian, then so is the
space $\overline{F}^n$.
 \end{thm}

 \begin{proof}~\par
 We prove firstly that if the covariant vector $b_i(x)$ is
Cartan-parallel  and the change is homothetic, then
$\overline{R}^i_{hjk}$ vanishes if and only if  $R^i_{hjk}$
vanishes.
  By  Corollary \ref{final-1},
 $\overline{R}^i_{hjk}=R^i_{hjk}+M^i_{hm}R^m_{jk}$. If $R^i_{hjk}=0$,  then
 $R^i_{jk}$=0 and hence $\overline{R}^i_{hjk}=0$. Conversely, if
 $\overline{R}^i_{hjk}=0$, then $R^i_{hjk}+M^i_{ht}R^t_{jk}=0$.
 By transvection  with $y^h$, we obtain $R^i_{jk}=0$ and hence
 $R^i_{hjk}=0$.

 Now, the result follows from the above fact and Theorem
 \ref{berwald}.
 \end{proof}

\begin{thm}\label{minkowski.2} Under a generalized $\beta$-conformal change,
 a Riemannian space
$F^n$ is transformed to a locally Minkowskian space
$\overline{F}^n$ if
  $b_{i}$ is Riemann-parallel, $\sigma$ is homothetic and $R^i_{hjk}$ vanishes.
\end{thm}
\begin{proof}~\par
Follows directly from Corollary \ref{cor1}.
\end{proof}

It is to be  noted  that Theorem \ref{landesberg}, Corollary
\ref{rlandesberg}, Theorem \ref{minkowski.1} and Theorem
\ref{minkowski.2} generalize various results of Shibata
\cite{r1.13}.

%%%%%%%%%%%%%%%%%%%%%%%%%%%%%%%%%%%%%%%%%%%%%%%%%%%%%%%%%%%%%%%%%%%%%%%%%%%%%%%%%
\Section{\textbf{Concluding remarks}}

In this paper, we have introduced a generalized  change, which
combines  both $\beta$-change  and conformal change in a general
setting. We have refered to this change as a generalized
$\beta$-conformal change. Many of the known Finsler changes in the
literatures may be obtained from the generalized $\beta$-conformal
change as special cases.

 We will mention some  interesting special cases. In these special cases, we  restrict ourselves to the difference tensor
$D^i_{jk}$ only.\\

 $\bullet$ When the generalized $\beta$-conformal change
(\ref{change}) is a  $\beta$-change: $\overline{L}=f(L, \beta)$,
the difference tensor (\ref{eqn.4}) takes the form:
\begin{eqnarray*}
    D^i_{jk}&=&\{(1/p){g}^{ir}-(s_0b^i+s_{-1}y^i)b^r-(s_{-1}b^i
    +s_{-2}y^i)y^r\}\{B_{jr}b_{0\mid k}+
    B_{kr}b_{0\mid j}\\&&-
      B_{jk}b_{0\mid r}
+F_{rk}Q_j+F_{rj}Q_k+E_{jk}Q_r+ p \,C_{jkm}
D^m_r+V_{jkm}D^m_r\\&&- p\,C_{rkm}D^m_j-V_{rkm}D^m_j-
p\,C_{rjm}D^m_k-V_{rjm}D^m_k\}.
\end{eqnarray*}
This is the case    studied by Shibata
\cite{r1.13}.\\

$\bullet$ When the generalized $\beta$-conformal change
(\ref{change}) is a $\beta$-conformal change:
$\overline{L}=e^\sigma L+\beta$, the difference tensor
(\ref{eqn.4}) takes the form:
 \begin{eqnarray*}
      D^i_{jk}&=&[\tau^{-1}g^{ir}-(1/\overline{L}\tau)(y^ib^r+y^rb^i)+\mu l^il^r][F_{rk}Q_j
       +F_{rj}Q_k+E_{jk}Q_r\\
       &&+(1/2)\Theta_{(j,k,r)}\{2\tau C_{jkm}
      D^m_r+2V_{jkm}D^m_r- K_{jk}\sigma_{r}-(e^\sigma/L)h_{jk}b_{0\mid
      r}\}],
      \end{eqnarray*}
  where     $\tau:=e^\sigma\frac{\overline{L}}{L}$ and $\mu:=\frac{L}{\tau \overline{L}^2}(Lb^2+\beta
  e^\sigma)$.\\
  This is the case  studied by  Abed \cite{r1.2}.\\

$\bullet$ When the generalized $\beta$-conformal change
(\ref{change}) is a generalized Randers  change:
$\overline{L}=L+\beta$, with $L$  Finslerian, the difference
tensor (\ref{eqn.4}) takes the form:
\begin{eqnarray*}
    D^i_{jk}
      &=&[\tau^{-1}g^{ir}-\frac{1}{\overline{L}\tau}(y^ib^r+y^rb^i)+\mu
l^il^r][
      F_{rk}Q_j+F_{rj}Q_k+E_{jk}Q_r\\
      &&+(1/2)\Theta_{(j,k,r)}\{2\tau C_{jkm}
      D^m_r+2V_{jkm}D^m_r-(1/L)h_{jk}b_{0\mid
      r}\}],
      \end{eqnarray*}
      This is the case studied by Matsumoto \cite{r1.8}, Tamim and
      Youssef  \cite{r1.15} and others.\\

$\bullet$ When the generalized $\beta$-conformal change
(\ref{change}) is a  Kropina  change: $\overline{L}=L^2/\beta$,
with $L $ Riemannaian,  the difference tensor (\ref{eqn.4})
  takes the form:
\begin{eqnarray*}
   D^i_{jk}&=&\frac{1}{2L^4b^2\beta^3}\left[L^2b^2g^{ir}-(L^2b^i-2\beta y^i)b^r
   +2(m^2y^i+\beta m^i)y^r\right]\\
   &&.[\beta L^2\left(F_{rk}(3L^2b_j-4\beta y_j)
      +F_{rj}(3L^2b_k-4\beta y_k)
      +E_{jk}(3L^2 b_r-4\beta y_r)\right)\\
    &&+(1/2)\Theta_{(j,k,r)}\{2\beta^5V_{jkm}D^m_r+4L^2(\beta^2h_{jk}+3 L^2m_jm_k)b_{0\mid
      r}\}].
\end{eqnarray*}
This is the case studied by Kropina \cite{r1.4}, Matsumoto \cite{r1.9},  Shibata \cite{r1.12} and others.\\

$\bullet$ When the generalized $\beta$-conformal change
(\ref{change}) is a conformal   change: $\overline{L}=e^\sigma L$,
the difference tensor (\ref{eqn.4}) takes the form:

\begin{eqnarray*}
    D^i_{jk}&=&\sigma_j\delta^i_k+\sigma_k\delta^i_j-\sigma^ig_{jk}+y_jC^i_{km}\sigma^m
    +y_kC^i_{jm}\sigma^m -y^iC_{jkm}\sigma^m
-\sigma_0C^i_{jk}\\
&&+L^2(C_{jkm}C^{mi}_r\sigma^r-C^i_{km}C^m_{jr}\sigma^r-
C^i_{jm}C^m_{kr}\sigma^r).
\end{eqnarray*}
This is the case studied by Hashiguchi
\cite{r1.6}, Izumi \cite{r1.7}, Youssef et al. \cite{r1.17} and others.\\

$\bullet$ When the generalized $\beta$-conformal change
(\ref{change}) is a  C-conformal (resp. h-conformal)   change:
$\overline{L}=e^\sigma L$, with  $\sigma$ enjoying the property
that
      $C^i_{jk}\sigma_i=0$ (resp. $C^i_{jk}\sigma_i=\frac{1}{n-1}C^i\sigma_ih_{jk}$),
the difference tensor (\ref{eqn.4}) takes the form:

$$D^i_{jk}=\sigma_j\delta^i_k+\sigma_k\delta^i_j-\sigma^ig_{jk}
-\sigma_0C^i_{jk}.$$
$$(resp.\quad D^i_{jk}=\sigma_j\delta^i_k+\sigma_k\delta^i_j-\sigma^ig_{jk}-\sigma_0C^i_{jk}
    +\frac{1}{n-1}C^r\sigma_r(y_jh^i_k+y_k h^i_j -y^ih_{jk}
-L^2C^i_{jk})).$$ This is the case studied by Hashiguchi
\cite{r1.6} (resp. Izumi \cite{r1.7}).

 \providecommand{\bysame}{\leavevmode\hbox
to3em{\hrulefill}\thinspace}
\providecommand{\MR}{\relax\ifhmode\unskip\space\fi MR }
\providecommand{\MRhref}[2]{%
  \href{http://www.ams.org/mathscinet-getitem?mr=#1}{#2}
} \providecommand{\href}[2]{#2}
{\vspace{-1cm}}


\begin{thebibliography}{1}

\bibitem{r1.1}
S.~H. Abed, \emph{Conformal $\beta$-changes in Finsler spaces},
Proc. Math. Phys. Soc. Egypt, \textbf{86} (2008), 79--89. ArXiv
No.: math. DG/0602404.

\bibitem{r1.2}
S.~H. Abed, \emph{Cartan connection associated with a
$\beta$-conformal change in
  Finsler geometry}, Tensor, N. S., \textbf{70} (2008), 146--158.
ArXiv No.: math. DG/0701491.

  \bibitem{r1.3}
X. Chen and Z. Shen, \emph{Randers metrics with special curvature
properties}, Osaka J. Math., \textbf{40} (2003), 87--101.



\bibitem{r1.5}
M. A. Eliopoulos, \emph{A generalized metric space for
electromagnetic theory}, Acad. Roy. Belg. Bull. Cl. Sci.,
\textbf{51} (1965), 986--995.


\bibitem{r1.6}
M.~Hashiguchi, \emph{On conformal transformation of Finsler
metrics},
  J. Math. Kyoto Univ., \textbf{16} (1976), 25--50.

\bibitem{Ingarden}
R. S. Ingarden, \emph{On the geometrically absolute representation
in the electron microscope, } Trav. Soc. Lett. Wroclaw, \textbf{B
45},  (1957), 60p.

\bibitem{r1.7}
H. Izumi, \emph{Conformal transformations of Finsler spaces, }
Tensor, N. S., \textbf{31} (1977), 33--41.

\bibitem{r1.4}
V. K. Kropina, \emph{On Projective two-dimentional Finsler spaces
with special metric, }
  Truday Sem. Vektor. Tenzor. Anal., \textbf{11} (1961), 277--292.

\bibitem{r1.9}
M.~Matsumoto, \emph{On C-reducible Finsler spaces, } Tensor, N.
S., \textbf{24} (1972), 29--37.

\bibitem{r1.8}
M.~Matsumoto, \emph{On Finsler spaces with Randers metric and special
  forms of important tensors}, J. Math. Kyoto Univ., \textbf{14} (1974),
  477--498.

\bibitem{r1.10}
R. Miron, \emph{The geometry of Ingarden spaces}, Rep. Math.
Phys., \textbf{54} (2004), 131--147.

\bibitem{r1.11}
G.~Randers, \emph{On the asymmetrical metric in the four-space of
general relativity}, Phys. Rev., (2) \textbf{59} (1941), 195--199.

\bibitem{r1.12}
C.~Shibata, \emph{On Finsler spaces with Kropina metric,} Rep.
Math. Phys., \textbf{13} (1978), 117--128

\bibitem{r1.13}
C.~Shibata, \emph{On invariant tensors of  $\beta$-changes of Finsler
  metrics}, J. Math. Kyoto Univ., \textbf{24} (1984), 163--188.

\bibitem{r1.14}
C. Shibata,  H. Shimada,  M. Azuma and H. Yasuda, \emph{On Finsler
spaces with Randers metric}, Tensor, N. S., \textbf{31} (1977),
219--226.

\bibitem{r1.15}
A. A. Tamim and Nabil L. Youssef, \emph{On generalized Randers
manifolds},  Algebras, Groups and Geometries, \textbf{16} (1999),
115--126.
  ArXiv No.: math. DG/0607572.

\bibitem{r1.16}
H. Yasuda and  H. Shimada,   \emph{On Randers spaces of scalar
curvature}, Rep.  Math. Phys., \textbf{11} (1977),  347--360.

\bibitem{r1.17}
Nabil L. Youssef, S. H. Abed and A. Soleiman, \emph{A global
theory of conformal Finsler geometry}, Tensor, N. S., \textbf{69}
(2008), 155--178. ArXiv No.: math. DG/0610052.

\bibitem{geo.}
Nabil L. Youssef, S. H. Abed and A. Soleiman, \emph{Concurrent $\pi$-vector fields and eneregy
$\beta$-change}, Accepted for publication  in Int. J. Geom. Meth.
Mod. Phys., (2009). ArXiv Number: 0805.2599v2 [math.DG].



\end{thebibliography}
\end{document}